\documentclass{article}
\usepackage[utf8]{inputenc}
\usepackage{algorithm2e}
\usepackage{mathtools}
\usepackage{amsthm}
\usepackage{cite}
\usepackage{comment}
\usepackage{authblk}
\usepackage{todonotes}
\usepackage{hyperref}
\usepackage[toc,page]{appendix}
\usepackage[dvipsnames]{xcolor}
\usepackage{amssymb}
\usepackage{newpxtext}

\newtheorem{lemma}{Lemma}
\newtheorem{theorem}{Theorem}
\newtheorem{proposition}{Proposition}

\begin{document}
\title{Almost balanced ordered biclique covering of graphs}
\author{Anand Babu, Ervin Ranjan, Maddipati Deshith Sai, Jatla Naga Sidhartha, Anagh Indu Suresh, Sreedhara Vishwas }
\affil{Department of Computer Science {\&} Engineering, \\ National Institute of Technology Calicut, Kozhikode, India}
\date{}
\maketitle
\begin{abstract}
Let $f(n,k)$ be the minimum size of a collection of bicliques such that (i) every edge of the complete graph $K_n$ is covered by at least one and at most $k$ bicliques in the collection, and (ii) for each edge $\{u,v\}$, the number of bicliques in which $u$ appears in the first class and $v$ in the second class differs by at most one from the number of bicliques in which $u$ appears in the second class and $v$ in the first class.

For $k=1$, $f(n,k)$ reduces to the biclique partition number of $K_n$, and the Graham-Pollak theorem gives $f(n,1)=n-1$. For $k=2$, $f(n,k)$ is the ordered biclique partition number of $K_n$, for which it is known that $c_1 n^{1/2} \le f(n,2) \le c_2 n^{1/2+o(1)}$ for some positive constants $c_1$ and $c_2$. In this note, we give almost tight bounds for $f(n,k)$ for fixed $k \ge 2$:
\[
(1+o(1))c_1(k)\cdot n^{\frac{1}{\lceil k/2\rceil+1}}
\le f(n,k)
\le (1+o(1))c_2(k)\cdot n^{\frac{1}{\lfloor k/2\rfloor+1}+o(1)},
\]
where $c_1(k)$ and $c_2(k)$ are positive constants.
\end{abstract}

\section{Introduction}
The \emph{biclique partition number} $bp(G)$ of a graph $G$ is the minimum number of complete bipartite graphs (bicliques) whose edge sets partition $E(G)$. In contrast, the minimum number of bicliques that cover $E(G)$ at least once is the \emph{biclique cover number} of $G$. Hansel \cite{hansel1964nombre} determined the exact biclique cover number of the complete graph, proving that $bc(K_n) = \lceil \log_2 n \rceil$. The celebrated Graham-Pollak theorem \cite{graham1971addressing, graham1972embedding} states that $bp(K_n) = n-1$. The original proof by Graham and Pollak uses Sylvester's law of inertia \cite{graham1972embedding}. Other proofs of the same were found by Tverberg\cite{tverberg1982decomposition},Peck\cite{peck1984new} and Vishwanathan \cite{vishwanathan2008polynomial} using linear algebraic methods. A combinatorial proof was given by Vishwanathan \cite{vishwanathan2013counting}.

Aharoni and Linial\cite{alon1986decomposition} posed the natural extension of this problem on the minimum size of the family of complete $r$-partite $r$-graphs required to partition the edge set of the complete
$r$-uniform hypergraph for $r > 2$. The minimum size of the collection of complete $r$-partite $r$-graphs that partitions the edge set of the complete $r$-uniform hypergraph on $n$ vertices is represented by $f_r(n)$. For $r=2$, $f_r(n)$ represents the minimum size of the biclique partition of $K_n$. For $r=3$, Alon \cite{alon1986decomposition} showed that $f_r(n)=n-2$. For general $r$, constructions due to Alon \cite{alon1986decomposition} and later Cioab\u{a}, K\"undgen, and Verstra\"ete \cite{cioabua2009decompositions} yield an upper bound on $f_r(n)$, while Alon \cite{alon1986decomposition} obtains a linear-algebraic lower bound. Asymptotically,
\(
\frac{2}{\binom{2\lfloor r/2 \rfloor}{\lfloor r/2 \rfloor}}(1+o(1))\binom{n}{\lfloor r/2 \rfloor}
\le f_r(n)
\le (1+ o(1))\binom{n}{\lfloor r/2 \rfloor}.
\) Improvements on the upper bound were provided by Leader, Mili\'cevi\'c, and Tan \cite{leader2017decomposing} and Babu and Vishwanathan \cite{babu2019bounds}. Indeed, it is shown in \cite{leader2018improved} that $f_r(n) \leq \frac{r}{2} (\frac{14}{15})^{r/4} (1+o(1)) \binom{n}{\lfloor r/2 \rfloor}$.

Several natural extensions were also studied. If $L$ is a list of positive integers and $G$ is a graph, the minimum number of bicliques that partition $E(G)$ such that each edge of $G$ is contained in exactly $l$ bicliques for some $l \in L$ is the \emph{$L$-bipartite covering number} of the graph $G$\cite{cioabua2013variations}. It is denoted by $bp_L(G)$. If $L=\{1\}$ then $bp_L(G)$ is same as $bp(G)$. For a list $L=\{1,\cdots,t\}$, $bp_{L}(G)$ is the bipartite covering number of order $t$ or the $t$-biclique covering number of $G$ \cite{alon1997neighborly, huang2012counterexample}. There are not many lists of constant size greater
than one for which the exact value of $bp_L(K_n)$ is known. Further bounds on list covering of graphs and hypergraphs can be found in \cite{babu2021multicovering, babu2022improved}\cite{leader2024odd}\cite{buchanan2024odd, buchanan2023odd}\cite{alon2023new}\cite{radhakrishnan2000depth} \cite{cheng2024exact}\cite{grytczuk2024neighborly} \cite{de1993minimum}.

A variant of list covering of graphs, named ordered biclique partitions was investigated by Shigeta and Amano \cite{Amano2014, ShigetaAmano2015} motivated by non-deterministic communication complexity. This relates graph decompositions with the size of boolean fooling sets. An \emph{ordered biclique partition} of the complete graph $K_n$ is a collection of bicliques in which each edge is covered at least once and at most twice, and any edge covered twice has opposite orientations across the two bicliques: each endpoint lies in the first class in one biclique and in the second class in the other. The minimum number of bicliques in the collection is the ordered biclique partition number denoted by $bp_{1.5}(K_n)$. Shigeta and Amano\cite{ShigetaAmano2015} proves an upper bound of $bp_{1.5}(K_n) \leq c_1 \cdot n^{1/2+o(1)}$ and the lower bound of $bp_{1.5}(K_n) \geq c_2 \cdot n^{1/2}$ is due to Alon\cite{alon1997neighborly}.

Let $[k]=\{1,\cdots,k\}$. An almost balanced ordered biclique covering of order $k$ of a graph $G$ is a collection of bicliques such that each edge $e \in E(G)$ is contained in at least one and at most $k$ bicliques, with the added restriction that if an edge $e=\{u,v\}$ is covered by $t$ bicliques, then the number of bicliques with $u$ in the first class and $v$ in the second class differs by at most one from the number of bicliques with $v$ in the first class and $u$ in the second class. The minimum number of bicliques in the collection is the almost balanced ordered biclique covering number of order $k$ of the graph $G$. It is denoted by $f(n,k)$.

In Section \ref{lower_bound}, the following lower bound of $f(n, k)$ is proved for $k\geq 2$.
\[
f(n,k) \geq (1+o(1)) \cdot  c_1(k) \cdot n^{\displaystyle \frac{1}{\lceil k/2 \rceil+1}}
\]
for a positive constant \(c_1(k)\).

In Section \ref{upper_bound}, we achieve the following upper bound for $f(n, k)$.

\[
f(n,k) \leq (1+o(1))\cdot c_2(k) \cdot n^{\displaystyle \frac{1}{\lfloor k/2 \rfloor+1}+o(1)}
\]
for a positive constant \(c_2(k)\).

\section{Lower Bound}\label{lower_bound}
In this section, we derive a lower bound for $f(n,k)$ using polynomial methods. The proof idea is to define a polynomial for each vertex, evaluate these polynomials on characteristic vectors induced by the color classes of the  almost balanced ordered biclique cover, and then combine linear independence and determine the dimension of the vector space generated by the polynomials.

\begin{theorem}\label{lower_bound_theorem}
For $k \geq 2$, the almost balanced ordered biclique covering of order \(k\)
\[
f(n,k) \geq (1+o(1)) \cdot  c_1(k) \cdot  n^{\displaystyle 1/(\lceil k/2 \rceil+1)}
\]
for a positive constant \(c_1(k)\).
\end{theorem}
   
\begin{proof}
Let $\mathcal{B}=\{B_1,\dots,B_d\}$ be an almost balanced ordered biclique covering of order $k$ and size $d$ of the complete graph on the vertex set $N=\{1,\dots,n\}$. 
For each $i\in\{1,\dots,d\}$, let $U_i$ and $W_i$ denote the first and second vertex classes, respectively, of $B_i$.

For each $i\in N$, define a polynomial
\(
P_i=P_i(x_1,\dots,x_d,y_1,\dots,y_d)
\)
as follows.

\[
\begin{aligned}[t]
P_i
&= 
\Bigg[
\prod_{r=1}^{\lceil k/2\rceil}
\left(
\sum_{p:\, i\in U_p} x_p-r
\right)
\Bigg]
\left(
\sum_{q:\, i\in W_q} y_q-1
\right).
\end{aligned}
\]
For each $j\in N$, let
\(
e_j=(w_{j1},\ldots,w_{jd},\,u_{j1},\ldots,u_{jd})
\)
be the $0$-$1$ vector in which $u_{jp}=1$ if $j\in U_p$ and $u_{jp}=0$ otherwise,
and $w_{jq}=1$ if $j\in W_q$ and $w_{jq}=0$ otherwise. We show that $P_i(e_j)=0$
when $i\neq j$, and that $P_i(e_i)\neq 0$ for all $i\in N$.

\medskip
\noindent

Fix distinct $i,j\in N$. Evaluate $P_i$ at 
$(x_1,\ldots,x_d,y_1,\ldots,y_d)=e_j$. Define
\(
X=\sum_{p:\, i\in U_p} x_p
\) and \(
Y=\sum_{q:\, i\in W_q} y_q.
\)
Then $X$ is the number of bicliques in which an edge $(i,j)$ appears such that $i$ lies in the first color class and $j$ lies in the second class, and $Y$ is the number of bicliques in which an edge $(i,j)$ appears such that $j$ lies in the first color class and $i$ lies in the second color class. Since $\mathcal{B}$ is an almost balanced ordered biclique covering of order $k$, the balancing condition for the edge $\{i,j\}$ yields $|X-Y|\le 1$. 
Consider the term,
\(
\Bigg[
    \prod_{r=1}^{\lceil k/2\rceil}(X-r)
\Bigg].
\)
If $1\le X\le \lceil k/2\rceil$, then one factor in the product is zero, and the term vanishes. Otherwise $X=0$. In this case, the almost balanced ordered biclique covering condition forces $(X,Y)=(0,1)$, and therefore $Y-1=0$. Therefore, $P_i(e_j)=0$ whenever $i\neq j$. 
\medskip
\noindent

Now fix $i\in N$. Evaluate $P_i$ at $e_i$. Since no biclique contains a vertex $i$ such that it lies in both classes of the same biclique. Hence under $e_i$ we have $X=0$ and $Y=0$ and the term equals
\(
\Bigg[
    \prod_{r=1}^{\lceil k/2\rceil}(-r)
\Bigg](-1)
=
\lceil k/2\rceil!
\neq 0.
\)
Therefore, $P_i(e_i)\neq 0$ for all $i\in N$..

\paragraph{}
Let $\overline{P}_i = \overline{P}_i(x_1,\ldots,x_d, y_1,\ldots,y_d)$ denote the multilinear polynomial derived from the standard expansion of $P_i$ as a sum of monomials, where each monomial of the form
\(
c \prod_{s \in S} x_s^{\delta_s} \prod_{t \in T} y_t^{\gamma_t},
\)
with all $\delta_s$ and $\gamma_t$ positive, is replaced by
\(
c \prod_{s \in S} x_s \prod_{t \in T} y_t.
\)
Note that as all variables $x_p$ and $y_q$ take values in $\{0,1\}$, we have
\(
P_i(x_1,\ldots,y_d) = \overline{P}_i(x_1,\ldots,y_d),
\)
since for any $\delta$, $0^{\delta} = 0$ and $1^{\delta} = 1$. Therefore, 
\(
    \overline{P}_i(e_j) = 0 \hspace{1mm}\text{for all } i \neq j,\hspace{1mm} \text{and} \hspace{1mm} \overline{P}_i(e_i) \neq 0.
\)

\paragraph{}
The polynomials $\overline{P}_i$ $(i \in N)$ are linearly independent. We prove this by contradiction. Let
\(
\sum_{i \in N} c_i \overline{P_i}(x_1,\ldots,y_d) = 0,
\)
for non-trival values of \(c_i\) for \(i \in N\). Then there exists an $l \in N$ such that $c_l \neq 0$. By substituting $(x_1,\ldots,y_d) = e_l$, we get $c_l = 0$, since \(
    \overline{P}_i(e_j) = 0 \hspace{1mm}\text{for all } i \neq j,\hspace{1mm} \text{and} \hspace{1mm} \overline{P}_i(e_i) \neq 0
\), contradiction. Therefore these polynomials are linearly independent.

\paragraph{}

The polynomial does not have monomials that contain both $x_i$ and $y_i$ for the same $i$.
The polynomial is in the space generated by all the monomials $\prod_{s \in S} x_s \prod_{t \in T} y_t$, where $S \subseteq [d]$ such that $0 \leq |S| \leq \lceil k/2\rceil$ and $T \subseteq [d]\setminus S$ such that $|T|\leq 1$. The total number of such pairs is $m= \sum_{i=0}^{\lceil k/2 \rceil} \binom{d}{i}$ +$\sum_{i=1}^{{\lceil k/2 \rceil}+1} \binom{d}{i-1}(d-i+1)$. Since the polynomials $\overline{P_i}$ for $i=1$ to $n$ are linearly independent members in this space it follows that $n \leq m$. Therefore, we have

\begingroup
\allowdisplaybreaks
\begin{align*}
    n &\leq \sum_{i=0}^{\lceil k/2 \rceil} \binom{d}{i}+\sum_{i=1}^{{\lceil k/2 \rceil}+1} \binom{d}{i-1}(d-i+1) \\
    &\leq \sum_{i=0}^{\lceil k/2 \rceil} \binom{d}{i}+\sum_{i=0}^{\lceil k/2 \rceil} \binom{d}{i}(d-i) \\
    &\leq \sum_{i=0}^{\lceil k/2 \rceil} \binom{d}{i}(d-i+1) \\
    &\leq (d+1) \sum_{i=0}^{\lceil k/2 \rceil} \binom{d}{i}\\
    & \leq (d+1)\cdot (1+\lceil k/2\rceil) \cdot \binom{d}{{\lceil k/2 \rceil}} && \text{(use $\sum_{i=0}^{m}\binom{n}{i}\le (m+1)\binom{n}{m}$)}\\
    & \leq (d+1) \cdot (1+\lceil k/2\rceil) \cdot \bigg(\frac{ed}{{\lceil k/2 \rceil}}\bigg)^{\lceil k/2 \rceil} && \text{(use $\binom{n}{r}\le (en/r)^r$)}\\
    & \leq 4\cdot\lceil k/2\rceil \cdot \bigg(\frac{e}{\lceil k/2 \rceil}\bigg)^{\lceil k/2 \rceil} d^{\lceil k/2 \rceil +1 } && \\
\end{align*}
\endgroup

Rearranging the last inequality, we get
\[
d^{\lceil k/2 \rceil+1}
\ge
\frac{n}{4 \cdot \lceil k/2\rceil}
\left(\frac{\lceil k/2\rceil}{e}\right)^{\!\lceil k/2\rceil}.
\]
Hence,
\[
d \ge
\left(\frac{n}{4 \cdot \lceil k/2\rceil}\right)^{\!1/(\lceil k/2\rceil+1)}
\left(\frac{\lceil k/2\rceil}{e}\right)^{\!\lceil k/2\rceil/(\lceil k/2\rceil+1)}.
\]

\end{proof}

\section{Upper Bound}\label{upper_bound}
In this section, we obtain an upper bound for $f(n,k)$ using a constructive argument. The proof idea is to represent vertices by strings that are subcubes of the $d$-dimensional hypercube. Lemma \ref{lemma_eq_coord} provides a large balanced family of binary vectors by fixing the Hamming weight to $\lfloor k/2 \rfloor$. The second lemma then boosts the size of the family by using a product construction that concatenates the balanced family of strings from the Lemma \ref{lemma_eq_coord} with a family of strings coming from the ordered biclique partition of $K_n$ \cite{ShigetaAmano2015}. Then for each coordinate $i\in[d]$, we obtain a biclique $B_i(U_i,W_i)$ by placing all vectors with $i$-th bit $0$ in $U_i$ and vectors with $i$-th bit $1$ in $W_i$. The resulting family of bicliques forms an almost balanced ordered biclique cover of order $k$.

\begin{lemma}\label{lemma_eq_coord}
There exists a set $T \subseteq \{0,1\}^d$ of size at least $\binom{d}{\lfloor k/2 \rfloor}$ such that for every distinct $x,y \in T$, the number of coordinates where $(x_i,y_i)=(0,1)$ equals the number of coordinates where $(x_i,y_i)=(1,0)$, and the number of coordinates in which $x$ and $y$ differ is at least $1$ and at most $k$.
\end{lemma}
\begin{proof}
    Let $m=\lfloor k/2 \rfloor$, and define
    \[
    T=\left\{x\in\{0,1\}^d:\sum_{i=1}^{d}x_i=m\right\}.
    \]

Consider any two distinct vectors $u,v\in T$. Let $p$ be the number of coordinates where $(u_i,v_i)=(0,1)$, and let $q$ be the number of coordinates where $(u_i,v_i)=(1,0)$. Since $u$ and $v$ both have exactly $m$ ones, every time $u$ has a $1$ and $v$ has a $0$ at some coordinate, there must be another coordinate where $u$ has a $0$ and $v$ has a $1$, since the total number of ones is $m$ in both vectors. Hence $p$ and $q$ are equal.

    Let $t$ be the number of coordinates in which $u$ and $v$ differ. Then
    \(
    t=p+q.
    \)
    Since $u\neq v$, there is at least one coordinate where they differ, so $t\ge 1$ (in fact, here $t\ge 2$ since $p=q$). Moreover, we have $p\le m$ and $q\le m$, and therefore
    \(
    t=p+q\le 2m=2\lfloor k/2\rfloor\le k.
    \)
    The size of $T$ is $\binom{d}{m}$.
\end{proof}

\begin{lemma}\label{product_lemma}
    There exists a set $\mathcal{H} \subseteq \{0,1,*\}^{2d}$ of size at least $c(k) \cdot \bigg(\frac{d^{\lfloor{k/2 \rfloor}+2}}{\log^2 d}\bigg)$ for some positive constant $c(k)$, such that for every distinct $x,y \in \mathcal{H}$, the number of coordinates where $(x_i,y_i)=(0,1)$ differs from the number of coordinates where $(x_i,y_i)=(1,0)$ is at most $1$, and the number of coordinates in which $x$ and $y$ differ is at least $1$ and at most $k$.
\end{lemma}
\begin{proof}
    For two families $\mathcal{F} \subseteq \mathcal{S}^p$ and $\mathcal{G} \subseteq \mathcal{S}^q$, let $\mathcal{F} \cdot \mathcal{G} = \{uv : u \in \mathcal{F}, v \in \mathcal{G}\}$ be the family of strings in $\mathcal{S}^{p+q}$ consisting of all possible concatenations of strings from $\mathcal{F}$ and $\mathcal{G}$. Clearly, we have $| \mathcal{F} \cdot \mathcal{G}| = |\mathcal{F}| \cdot |\mathcal{G}|$.

    Let $\mathcal{F}\subseteq \{0,1\}^d$ be the family from Lemma \ref{lemma_eq_coord}, and let $\mathcal{G}\subseteq \{0,1,*\}^d$ be the family of strings induced by the ordered biclique partition of $K_n$. From the construction in \cite{ShigetaAmano2015}, we have $|\mathcal{G}| \geq \frac{d^2}{\log^2 d}$. Therefore, we have 
    
    \begin{align*}
        |\mathcal{F} \cdot \mathcal{G}| \geq \binom{d}{\lfloor k/2 \rfloor} \cdot \frac{d^2}{\log^2 d} \geq \bigg(\frac{d}{\lfloor{k/2 \rfloor}}\bigg)^{\lfloor{k/2 \rfloor}}\cdot \frac{d^2}{\log^2 d} \geq \bigg(\frac{1}{\lfloor{k/2 \rfloor}}\bigg)^{\lfloor{k/2 \rfloor}} \cdot \bigg(\frac{d^{\lfloor{k/2 \rfloor}+2}}{\log^2 d}\bigg)
    \end{align*}

Consider two distinct strings $r,s \in \mathcal{F}\cdot \mathcal{G}$ where $r=uv$ and $s=xy$, such that $u,x\in \mathcal{F}$ and $v,y\in \mathcal{G}$. For $a,b\in\{0,1,*\}^{2d}$, define
\[
N_{01}(a,b)=|\{i:(a_i,b_i)=(0,1)\}|,\qquad
N_{10}(a,b)=|\{i:(a_i,b_i)=(1,0)\}|.
\]
Then $N_{01}(r,s)=N_{01}(u,x)+N_{01}(v,y)$ and
$N_{10}(r,s)=N_{10}(u,x)+N_{10}(v,y)$. Therefore, we have $N_{01}(r,s)-N_{10}(r,s)
=\big(N_{01}(u,x)-N_{10}(u,x)\big)+\big(N_{01}(v,y)-N_{10}(v,y)\big)$.
By Lemma \ref{lemma_eq_coord}, for any $u,x\in \mathcal{F}$ we have
\(
N_{01}(u,x)-N_{10}(u,x)=0.
\)
Since the strings in $\mathcal{G}$ are induced by the ordered biclique partition of $K_n$, for any $v,y\in \mathcal{G}$,
\(
\left|N_{01}(v,y)-N_{10}(v,y)\right|\le 1.
\)
Therefore,
\(
\left|N_{01}(r,s)-N_{10}(r,s)\right|\le 1.
\)

Define $D(a,b)=|\{i:\{a_i,b_i\}=\{0,1\}\}|$. It denotes the number of coordinates in which one entry is $0$ and the other is $1$. Then, $D(r,s)=D(u,x)+D(v,y).$

Consider the following cases.

\emph{Case 1: $u=x$ and $v\neq y$.} Then
$D(u,x)=0$ and, since $\mathcal{G}$ is the family of strings induced by the ordered biclique partition of $K_n$,
\(
1\le D(v,y)\le 2,
\)
so
\(
1\le D(r,s)\le 2.
\)

\emph{Case 2: $u\neq x$ and $v=y$.} Then $D(v,y)=0$ and, from Lemma \ref{lemma_eq_coord}, we have
\(
1\le D(u,x)\le 2\lfloor k/2\rfloor,
\)
so
\(
1\le D(r,s)\le 2\lfloor k/2\rfloor.
\)

\emph{Case 3: $u\neq x$ and $v\neq y$.} Then
\(
1\le D(u,x)\le 2\lfloor k/2\rfloor\) and
\(1\le D(v,y)\le 2,
\)
which gives
\(
2\le D(r,s)=D(u,x)+D(v,y)\le 2\lfloor k/2\rfloor+2.
\)

Therefore, we have $\left|N_{01}(r,s)-N_{10}(r,s)\right|\le 1$ and $1\le D(r,s)\le 2\lfloor k/2\rfloor+2$.

\end{proof}

\begin{theorem}\label{upper_bound_theorem}
    The almost balanced ordered biclique covering number of order $k$, 
    \[
    f(n,k) \leq (1+o(1))\cdot c_2(k) \cdot n^{\displaystyle \frac{1}{\lfloor k/2 \rfloor+1}+o(1)}
    \]
for a positive constant $c_2(k)$.
\end{theorem}
\begin{proof}
    Let $\mathcal{H}$ be the family of \(2d\)-tuple vectors $v = \langle v_1,\cdots,v_{2d} \rangle$ as in Lemma \ref{product_lemma}. Consider the following family of bicliques $\mathcal{B}=\{B_1(U_1,W_1),\cdots, B_{2d}(U_{2d},W_{2d})\}$, where $U_i$ and $W_i$ denote the first and second color classes respectively of the biclique $B_i \in \mathcal{B}$, for $i \in \{1,\cdots,2d\}$ such that $U_i$ contains all vectors $v$ such that $v_i = 0$, $W_i$ contains all vectors $v$ such that $v_i = 1$. All the vectors with $v_i=*$, does not occur in either $U_i$ and $W_i$.

    Consider any two distinct vectors $x,y \in \mathcal{H}$. Suppose an edge $x,y$ appears exactly in $t$ bicliques in $\mathcal{B}$. As mentioned in the Lemma \ref{product_lemma}, the difference in number of coordinates where $(x_i,y_i)=(0,1)$ and $(x_i,y_i)=(1,0)$ is at most one. Equivalently, the difference in number of bicliques in $\mathcal{B}$ that has $\{x \in U_i, y \in W_i\}$ and $\{x \in W_i, y \in U_i\}$ is at most one. Moreover, by Lemma \ref{product_lemma}, the number of coordinates in which $x$ and $y$ such that $\{x_i,y_i\}=\{0,1\}$ is at least $1$ and at most $2\lfloor k/2 \rfloor + 2$. This translates to $t$ being at least one and at most $2\lfloor k/2 \rfloor + 2$. Therefore, we have an almost balanced ordered covering of order $2\lfloor k/2 \rfloor + 2$ of a complete graph on $\bigg(\frac{1}{\lfloor{k/2 \rfloor}}\bigg)^{\lfloor{k/2 \rfloor}} \cdot \bigg(\frac{d^{\lfloor{k/2 \rfloor}+2}}{\log^2 d}\bigg)$ using $2d$ bicliques. Rescaling, we have 

    \begingroup
    \allowdisplaybreaks
    \begin{align*}
        n &\geq \bigg(\frac{1}{\lfloor{k/2 \rfloor}}\bigg)^{\lfloor{k/2 \rfloor}} \cdot \bigg(\frac{d^{\lfloor{k/2 \rfloor}+2}}{\log^2 d}\bigg)\\
        &\geq \frac{1}{\left(\frac{K}{2}-1\right)^{\frac{K}{2}-1}} \cdot \frac{(D/2)^{\frac{K}{2}+1}}{\log^2(D/2)} \qquad \text{(Simplifying with } K=2\lfloor k/2 \rfloor+2,\ D=2d\text{).}\\
        &\geq \frac{1}{4\cdot K^{\frac{K}{2}-1}} \cdot \frac{D^{\frac{K}{2}+1}}{\log^2 D}.
    \end{align*}
    \endgroup

    This gives the required upper bound for even values of order $K$. Since an almost balanced ordered covering of order $K-1$ is also an almost balanced ordered covering of order $K$, we have $f(n,K) \leq f(n,K-1)$, which provides the required bound for odd values of $K$.
\end{proof}

\bibliographystyle{plain}
\bibliography{refs}

\end{document}